\documentclass{article}
\usepackage{amssymb,amsmath}
\usepackage{graphicx}

\def\qed{\hfill {\hbox{${\vcenter{\vbox{               
   \hrule height 0.4pt\hbox{\vrule width 0.4pt height 6pt
   \kern5pt\vrule width 0.4pt}\hrule height 0.4pt}}}$}}}

\newtheorem{theorem}{Theorem}
\newtheorem{definition}{Definition}

\newenvironment{proof}[1][Proof]{\smallskip\noindent{\bf #1.}\quad}{\qed\par\medskip}

\date{}

\title{\Large \textbf{The Signed Weighted Resolution Set is Not a Complete Pseudoknot Invariant}}

\author{
Allison Henrich\footnote{henricha@seattleu.edu, Seattle University, Seattle, WA 98122, United States}\hspace{1cm}
Slavik Jablan\footnote{sjablan@gmail.com, The Mathematical Institute, Belgrade, 11000, Serbia}\hspace{1cm}
Inga Johnson\footnote{ijohnson@willamette.edu, Willamette University, Salem, OR 97301, United States}
}

\begin{document}

\maketitle

\begin{abstract}
When the signed weighted resolution set was defined as an invariant of pseudoknots, it was unknown whether this invariant was complete. Using the Gauss-diagrammatic invariants of pseudoknots introduced by Dorais et al, we show that the signed were-set cannot distinguish all non-equivalent pseudoknots. This goal is achieved through studying the effects of a flype-like local move on a pseudodiagram.
\end{abstract}
\bigskip


Keywords: pseudoknot, were-set

\section{Introduction}
\subsection{Pseudodiagrams, Pseudoknots and Gauss Diagrams}

Pseudodiagrams are knot or link diagrams where some of the crossing information is missing.  Where there is missing information, instead of a crossing with clearly marked over- and under-strands, a {\em precrossing} or double-point of the curve appears in the diagram.   Pseudodiagrams of spatial graphs, knots and links were first introduced as potential models for biological applications by Hanaki~\cite{hanaki}. Pseudoknots were first defined in~\cite{pseudoknot} as equivalence classes of pseudodiagrams up to rigid vertex isotopy and a collection of natural Reidemeister moves.  This collection of moves includes the classical Reidemeister (R) moves and a number of additional pseudo-Reidemeister (PR) moves as seen in Figure~\ref{Rmoves}.

\begin{figure}[htbp]
\begin{center}
\includegraphics[scale=.75]{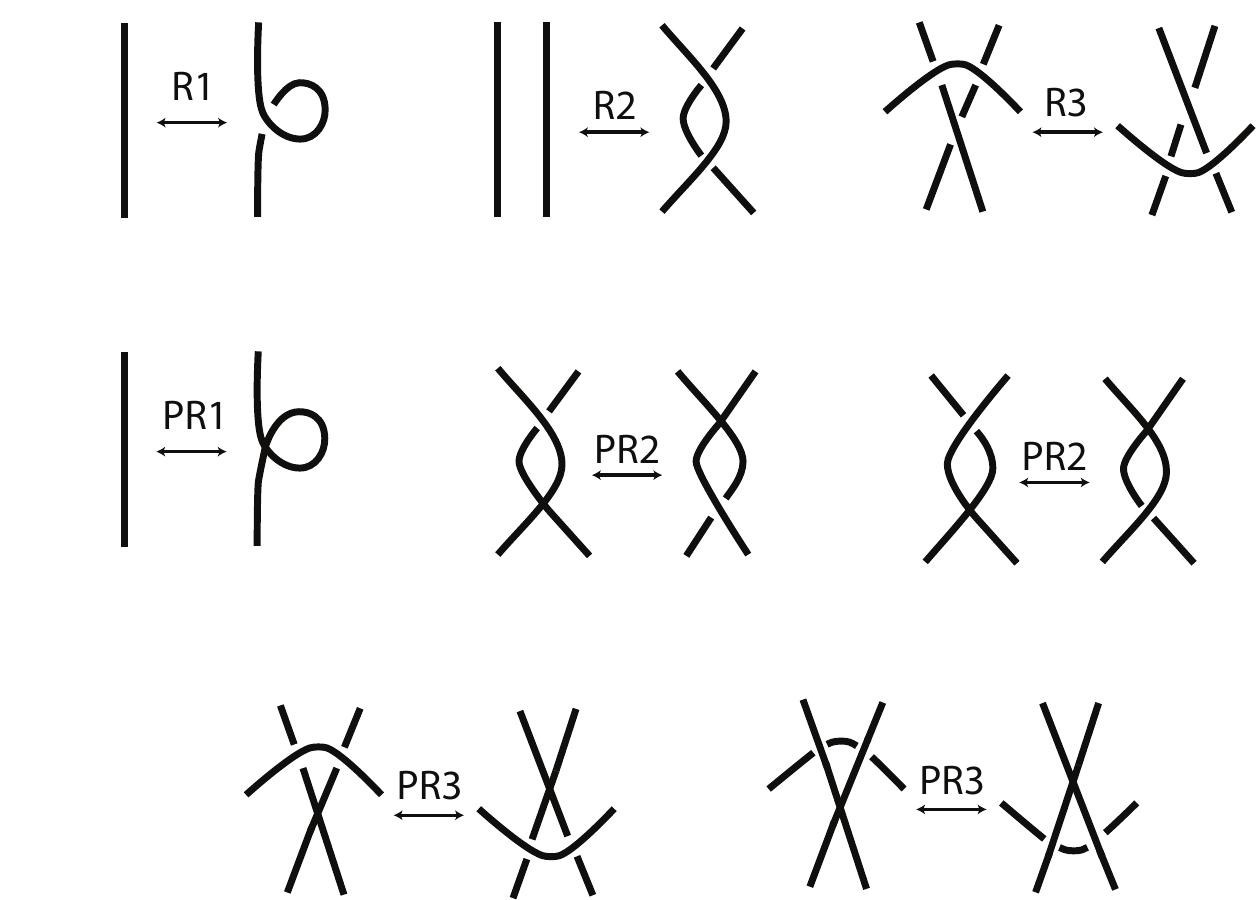}
\caption{{\bf Classical and Pseudo-Reidemeister moves}}
\label{Rmoves}
\end{center}
\end{figure}

The most familiar representation of a given knot is that of a knot diagram which is a shadow of the knot decorated with crossing information at transverse intersection points. An alternate and sometimes more useful presentation is a Gauss diagram which consists of a core circle oriented counterclockwise (drawn to represent the entire curve of the oriented knot) together with the pre-image of double-points connected by chords along the circle.  The crossing information is indicated on the chord by an arrow pointing from the over-strand to the under-strand and a sign on the chord specifying whether the crossing is left or right handed.  See Figure~\ref{G-example}.

\begin{figure}[htbp]
\begin{center}
\includegraphics[scale=.6]{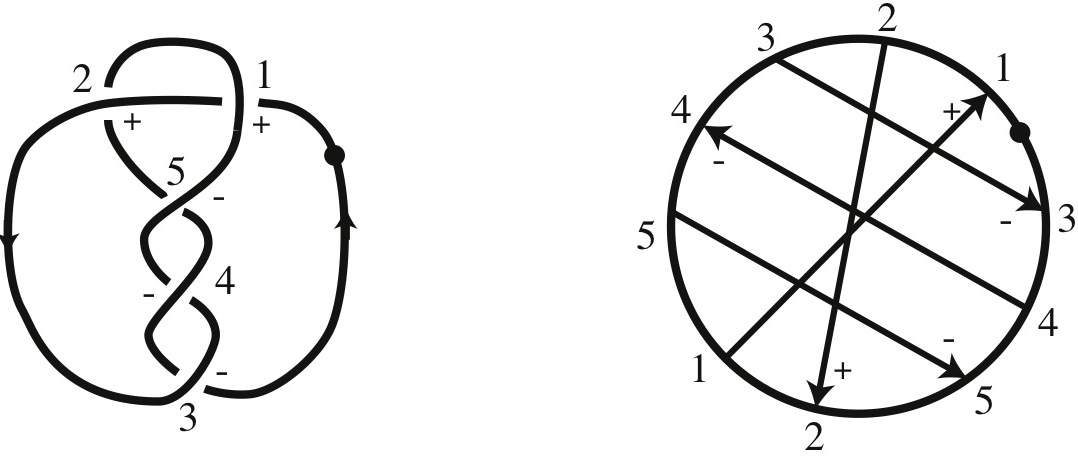}
\caption{{\bf A knot diagram and its corresponding Gauss diagram}}
\label{G-example}
\end{center}
\end{figure}

In~\cite{gauss}, the definition of a Gauss diagram was extended to pseudoknots as follows.  All classical crossings in a pseudoknot are represented in the Gauss diagram by a standard chord decorated with an arrow and a crossing sign.  A precrossing is represented by a bold or thicker chord.  We must take care, however, to indicate the proper `handedness' of the precrossing.  Figure~\ref{precrossing-ab} indicates a general precrossing and its decorated arrow within the Gauss diagram, and Figure~\ref{pseudoG-example} gives an example of a pseudoknot and its Gauss diagram.  Notice that in the Gauss diagram, the precrossing arrow points in the same direction as the classical arrow would point if the precrossing were resolved positively.

 \begin{figure}[htbp]
\begin{center}
\includegraphics[scale=0.6]{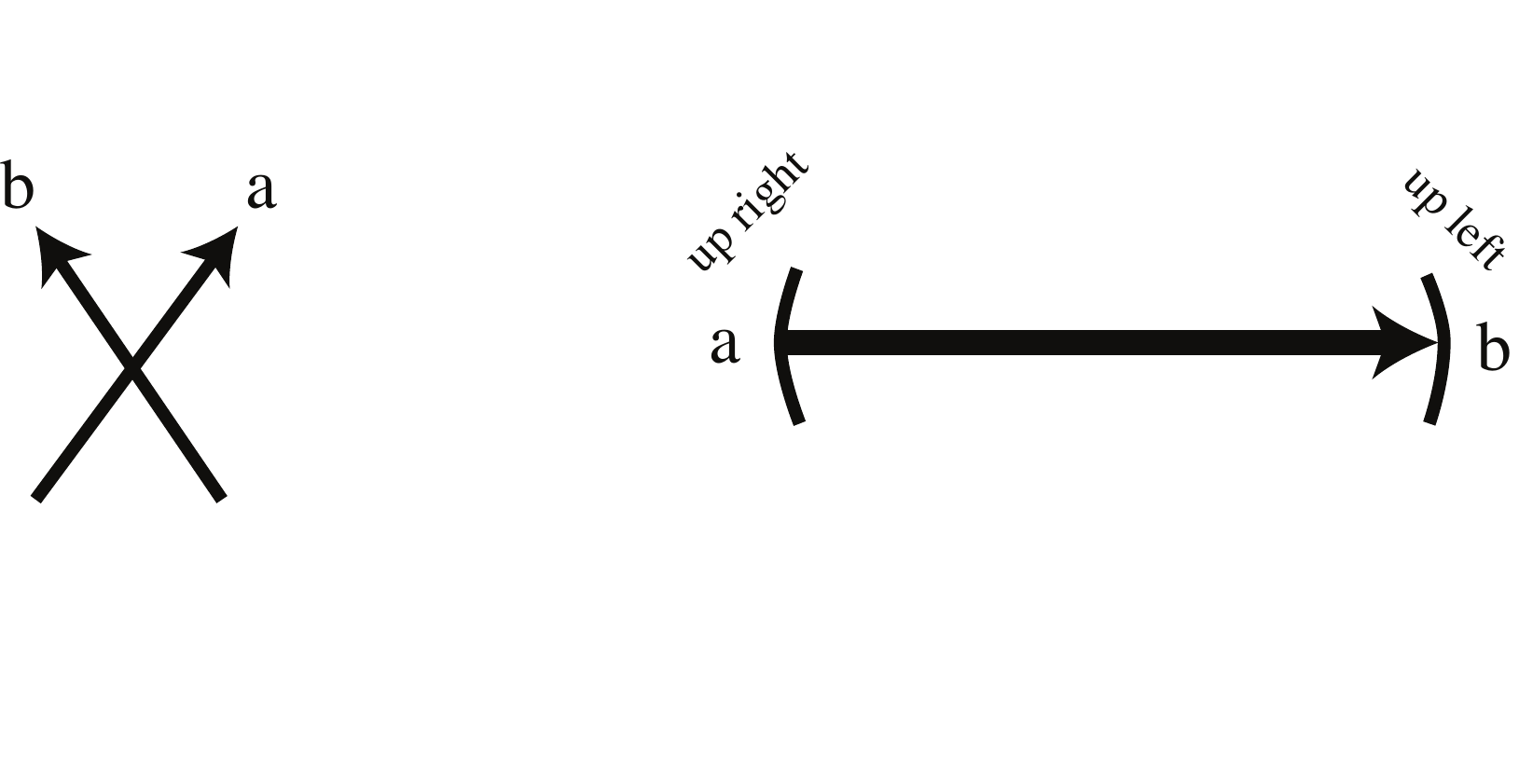}

\vspace{-.7in}
\caption{{\bf A precrossing and a subsection of its Gauss diagram}}
\label{precrossing-ab}
\end{center}
\end{figure}

 \begin{figure}[htbp]
\begin{center}
\includegraphics[scale=.6]{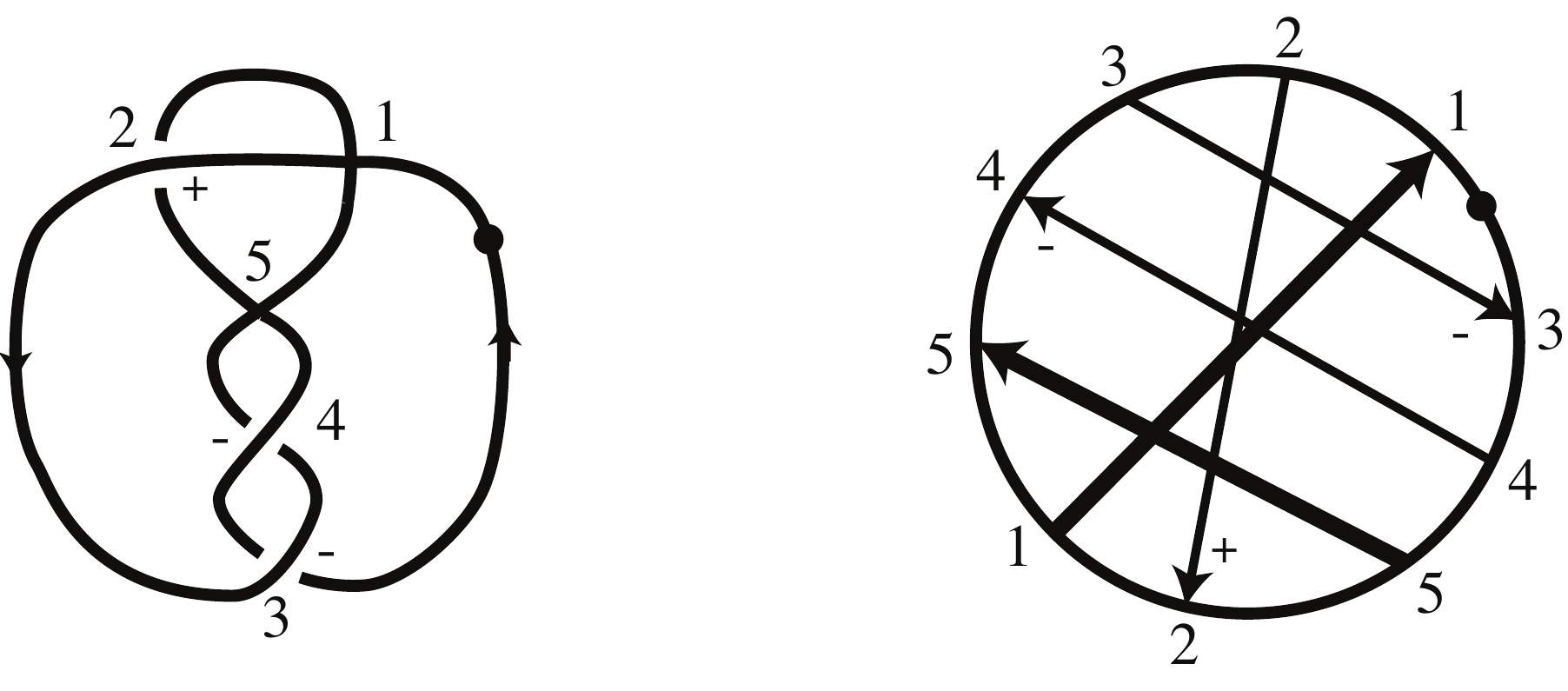}
\caption{{\bf A pseudoknot diagram and its corresponding Gauss diagram}}
\label{pseudoG-example}
\end{center}
\end{figure}

We refer the reader to~\cite{gauss} for illustrations of Gauss-diagrammatic versions of the R and PR moves. We note that not all Gauss diagrams correspond to pseudoknots, but the set of Gauss diagrams corresponds to a broader collection of virtual pseudoknots.

\subsection{Invariants of Pseudoknots}

There are two powerful pseudoknot invariants we focus our attention on in this paper, the signed were-set, introduced in~\cite{pseudoknot}, and the Gauss-diagrammatic invariant $\mathcal{I}(P)$, defined in~\cite{gauss}. We will use the latter invariant to show that the former invariant is incomplete.

\begin{definition}  The {\em  signed weighted resolution set} (or {\em were-set}) of a pseudoknot $P$ is the set of ordered pairs $(K, p_K)$ where $K$ is the knot type of a resolution of $P$ and $p_K$ is the probability that $K$ is obtained from $P$ by a random choice of crossing information, assuming that positive and negative crossings are equally likely. In this set, knots and their mirror images are treated as distinct unless such a pair of knots are actually topologically equivalent. \end{definition}

\begin{definition}
Consider a Gauss diagram of a (virtual or classical) pseudoknot, $G$. Define a map $\mathcal{I}(G)$ as follows.
\begin{enumerate}
\item Replace with chords all arrows in $G$ that are associated to precrossings. (I.e., delete all arrowheads on precrossing arrows.) These chords will be called {\em prechords}.
\item Decorate each prechord $c$ with the integer value $i(c)$, where $i(c)$ is the sum of the signs of the classical arrows that intersect $c$.
\item Delete all classical arrows.
\item Delete any prechords $c$ that have adjacent endpoints and $i(c)=0$.
\end{enumerate}
The codomain of $\mathcal{I}$ is the set of all chord diagrams such that each chord is decorated with an integer. We refer to this set as $\mathbb{Z}\mathcal{C}$.
\end{definition}

We illustrate this definition with an example of a virtual pseudoknot $P$ and its corresponding decorated chord diagram $\mathcal{I}(P)$ in Figure~\ref{I_example}.

\begin{figure}[htbp]
\begin{center}
\includegraphics[scale=.25]{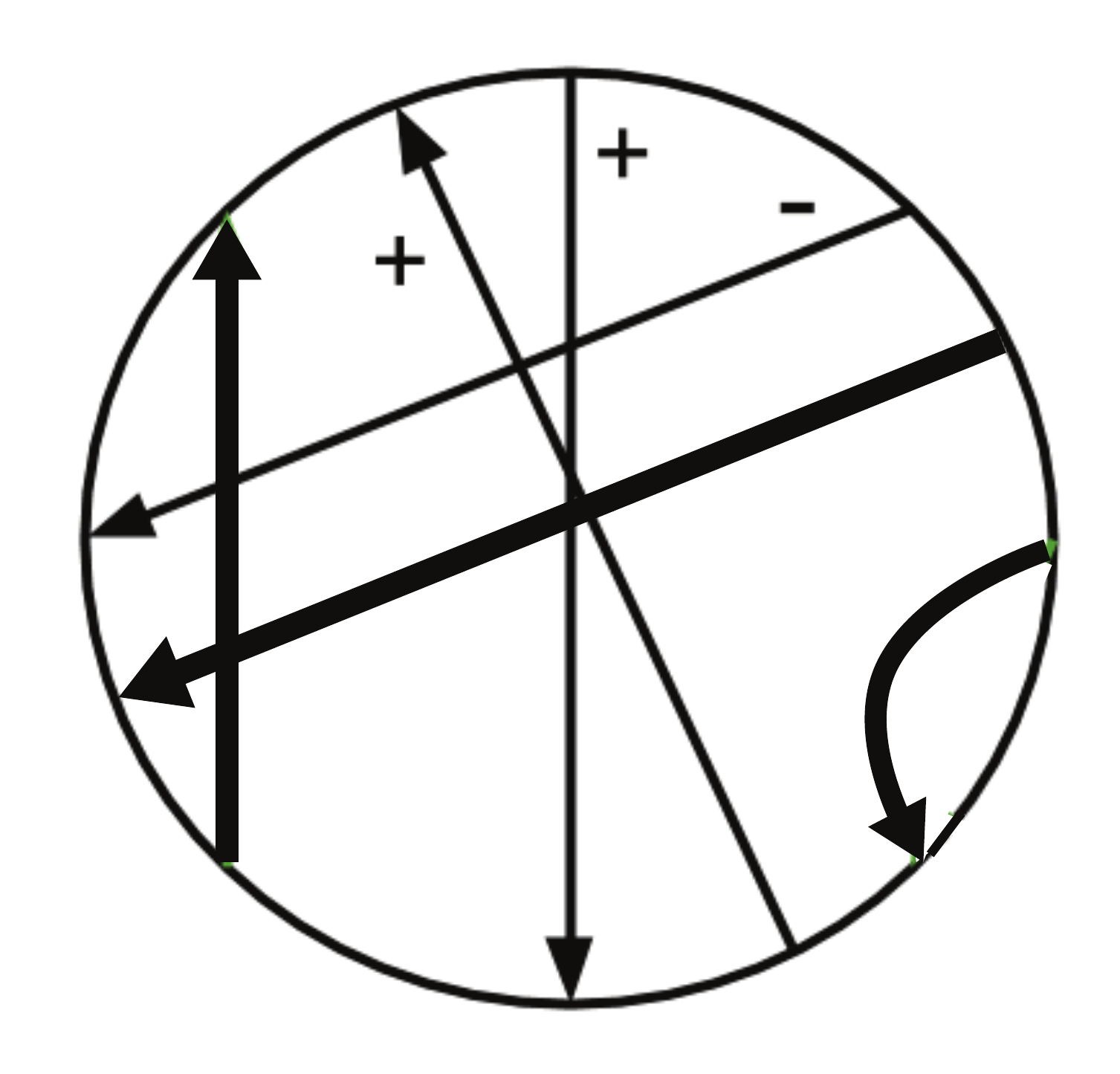}\hspace{1cm}\includegraphics[scale=.25]{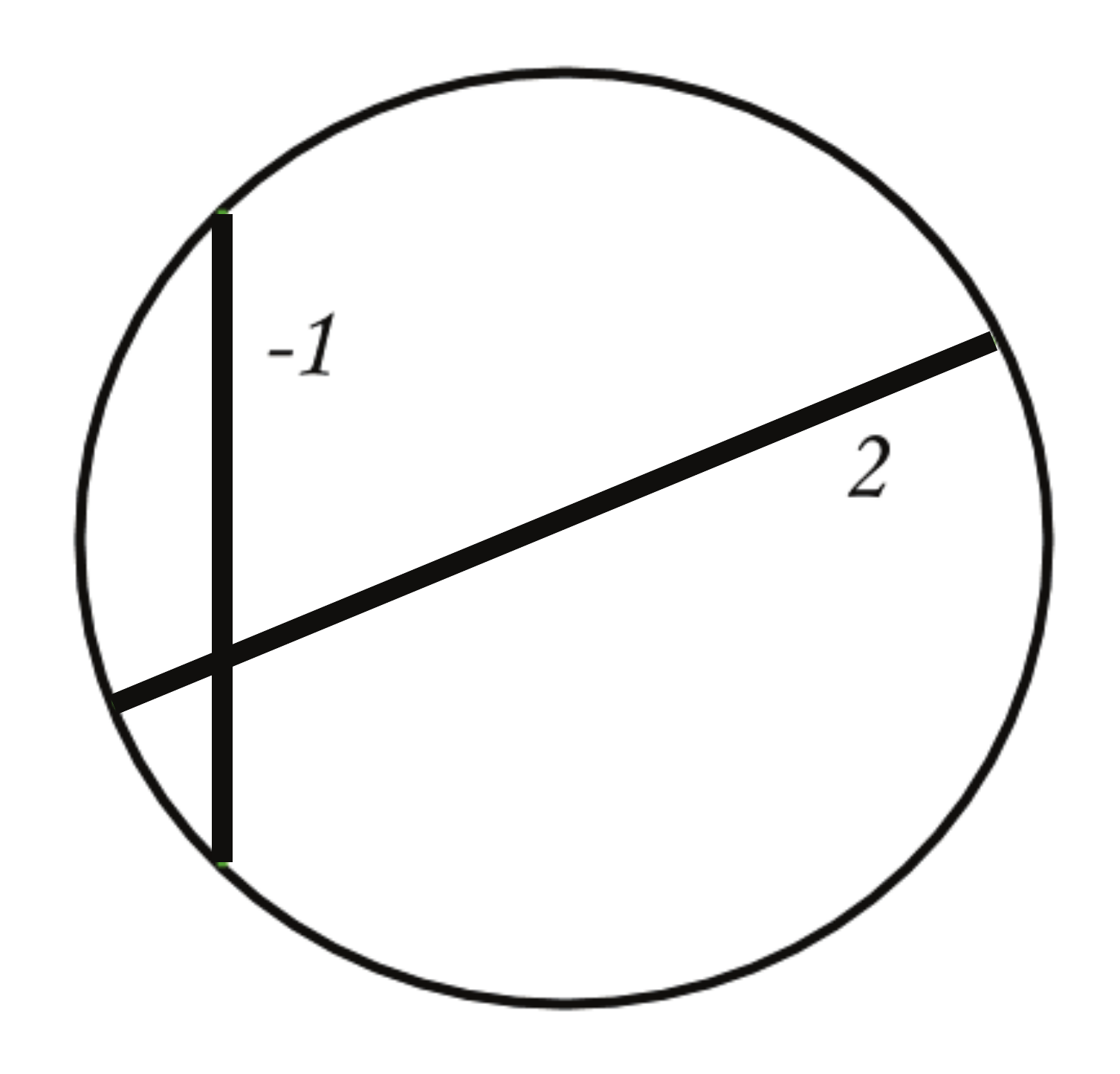}
\caption{{\bf A Gauss diagram for a virtual pseudoknot and its image under the map $\mathcal{I}$}}
\label{I_example}
\end{center}
\end{figure}

\section{Shadow Flypes and the Were-Set}
\subsection{The Were-Set is Incomplete}

To show that the were-set is incomplete, consider the pair of shadows that define two pseudoknots in Figure~\ref{counterexample}. We'll refer to these shadows as $P_1$ and $P_2$. 

\begin{figure}[htbp]
\begin{center}
\includegraphics[scale=.35]{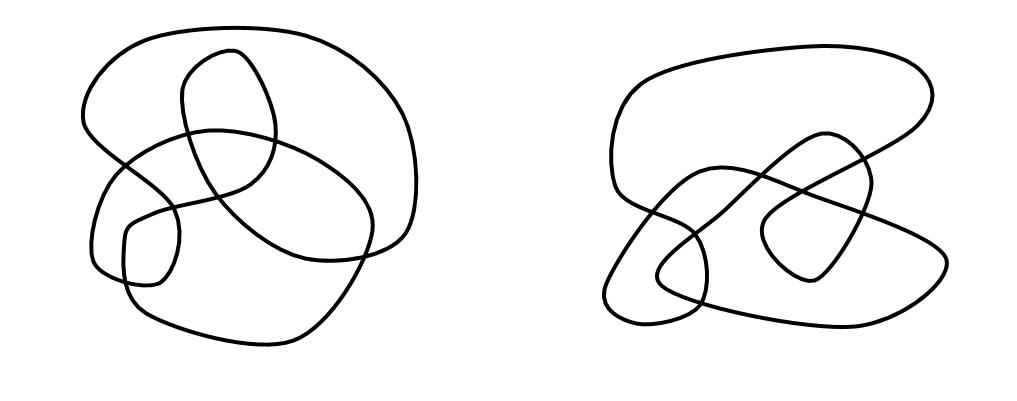}
\caption{{\bf A pair of nonequivalent pseudoknots, $P_1$ and $P_2$.}}
\label{counterexample}
\end{center}
\end{figure}

By construction, $P_1$ and $P_2$ are related by a move we will call a {\em shadow flype}, pictured in Figure~\ref{flype}. T is the shadow of a tangle in this figure. (This means that none of the crossings in the pseudotangle have been determined.)

\begin{figure}[htbp]
\begin{center}
\includegraphics[height=.8in]{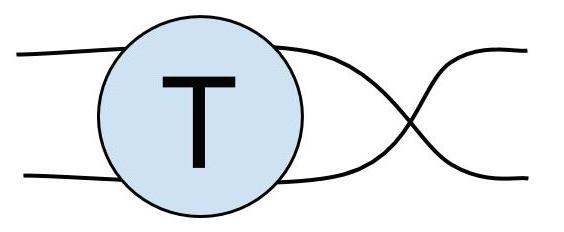}\includegraphics[height=.8in]{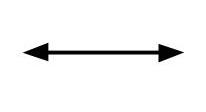}\includegraphics[height=.8in]{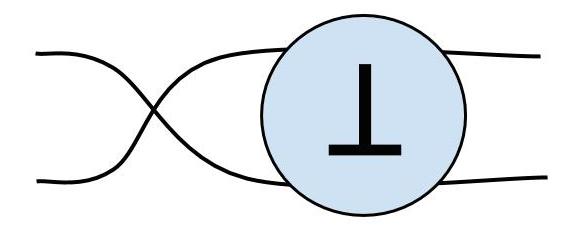}
\caption{{\bf A shadow flype.}}
\label{flype}
\end{center}
\end{figure}

In Figure~\ref{flypegauss}, we see the values of $\mathcal{I}$ for $P_1$ and $P_2$. Because any equivalences between the values of $\mathcal{I}$ must preserve the cyclic ordering of the endpoints of chords in the chord diagrams, the values of $\mathcal{I}$ are distinct. Thus, $P_1$ and $P_2$. are distinct.

\begin{figure}[htbp]
\begin{center}
\includegraphics[height=1.5in]{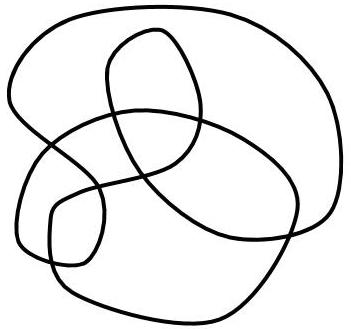}\includegraphics[height=1.7in]{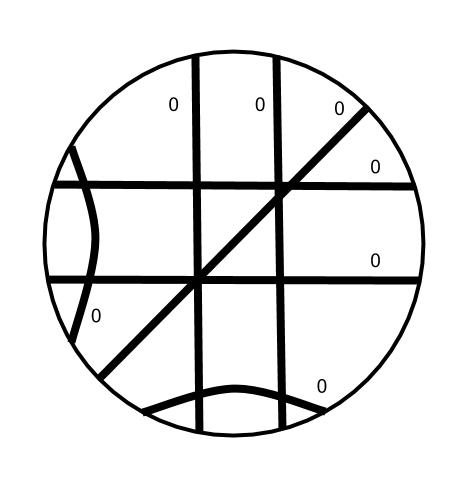}
\includegraphics[height=1.4in]{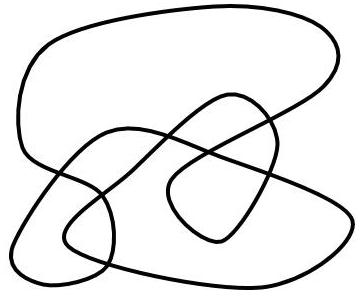}\includegraphics[height=1.65in]{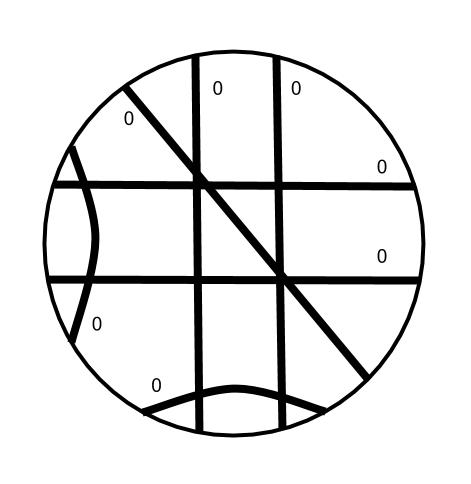}
\caption{{\bf Shadows $P_1$ and $P_2$ together with $\mathcal{I}(P_1)$ and $\mathcal{I}(P_2)$.}}
\label{flypegauss}
\end{center}
\end{figure}

What is interesting about this example is that for both pseudoknots, $P_1$ and
$P_2$, the signed were-sets are the same. Computed by {\em LinKnot}~\cite{4}, they are:

$$\{\{0_1,72\},\{-3_1,10\},\{3_1,10\},\{4_1,20\},\{-5_1,1\},\{-5_2,2\},$$ $$\{5_1,1\},\{5_2,2\},\{6_2,2\},\{6_1,2\},\{-6_1,2\},\{-6_2,2\},\{-7_7,1\},\{7_7,1\}\}.$$

(Note: the mirror image of knot $K$ is denoted by $-K$.)

In fact, it is true in general that the were-set is unchanged by the shadow flype.

\begin{theorem}\label{flype_thm} Suppose $P_1$ and $P_2$ are related by a local shadow flype move. Then the were-sets of $P_1$ and $P_2$ are identical.
\end{theorem}

\begin{proof} Suppose $P_1$ and $P_2$ are related by a local shadow flype move. Then there is a one-to-one correspondence between resolutions of $P_1$ and resolutions of $P_2$, where a knot diagram $D_1$ obtained by resolving precrossings of $P_1$ is paired with a resolution $D_2$ of $P_2$ such that $D_1$ and $D_2$ are related by a single (classical) flype. Thus, $D_1$ and $D_2$ represent equivalent knots. In effect, for any resolution of $P_1$, we can resolve the crossings of $P_2$ such that all crossings not involved in the shadow flype agree, and crossings involved in the shadow flype are chosen so that the resulting diagram is related to the resolution of $P_1$ by a classical flype. 
The existence of such a correspondence establishes our desired result.
\end{proof}

Theorem~\ref{flype_thm}, together with our example provide the proof of our main result:

\begin{theorem} The signed were-set is not a complete invariant of pseudoknots.
\end{theorem}

\subsection{The Effect of a Shadow Flype \& Examples}

We just saw for a specific example how the shadow flype affects a pseudodiagram, and we used the Gauss-diagrammatic $\mathcal{I}$ invariant to illustrate that this transformation changed the pseudoknot type. We'd like to explore these effects more generally. How does a shadow flype transform the Gauss diagram of a pseudodiagram? Secondly, how can the pair of examples in Figure~\ref{counterexample} be generalized to an infinite family of examples?

In~\cite{Soulie}, Souli\'{e} determines the effect on a Gauss diagram of performing a flype move on an ordinary knot. We can immediately derive from his work the effect of a shadow flype on the Gauss diagram $G$ of a pseudoknot $P$. There are two possibilities for how the Gauss diagram of a pseudoknot might be affected, illustrated in Figure~\ref{flypechord}. These diagrams should be interpreted as follows: in each horizontal (vertical) shaded region, there may be chords whose endpoints both lie within that region. Arrows are omitted here since the direction of precrossing arrows does not affect the value of $\mathcal{I}(G)$. Thus, we will limit our focus to considering shadow flypes on the level of chord diagrams.

\begin{figure}[htbp]
\begin{center}
\includegraphics[height=1.5in]{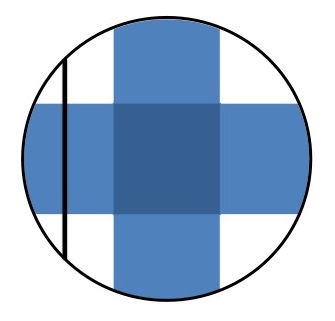}\includegraphics[height=1.5in]{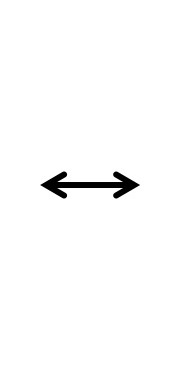}\includegraphics[height=1.5in]{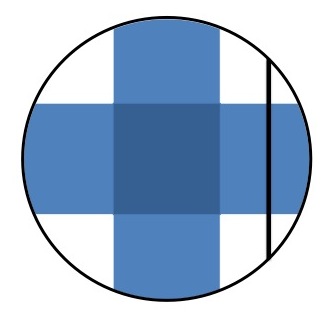}\\
Type I\\
\includegraphics[height=1.5in]{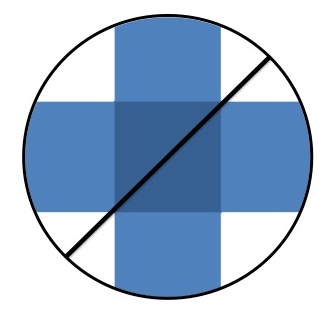}\includegraphics[height=1.5in]{arrow.jpeg}\includegraphics[height=1.5in]{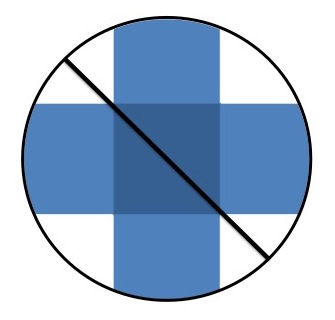}\\
Type II
\caption{{\bf Two types of flype effects on Gauss diagrams.}}
\label{flypechord}
\end{center}
\end{figure}

To understand how the diagrams in Figure~\ref{flypechord} apply to an example, we provide the chord diagrams for shadows $P_1$ and $P_2$ in Figure~\ref{flypegauss2}. As we can see, this pair of examples illustrates a Type II flype. It is not difficult to prove that this pair of examples is the smallest counterexample to the signed were-set conjecture that can be obtained from a shadow flype. This follows from two facts: (1) in order for a chord diagram to be obtainable from a classical pseudoknot, each chord must intersect an even number of chords, and (2) any two chord diagrams that are related by a planar isotopy are equivalent.

\begin{figure}[htbp]
\begin{center}
\includegraphics[height=1.5in]{PreflypeShadow.jpg}\hspace{.3in}\includegraphics[height=1.5in]{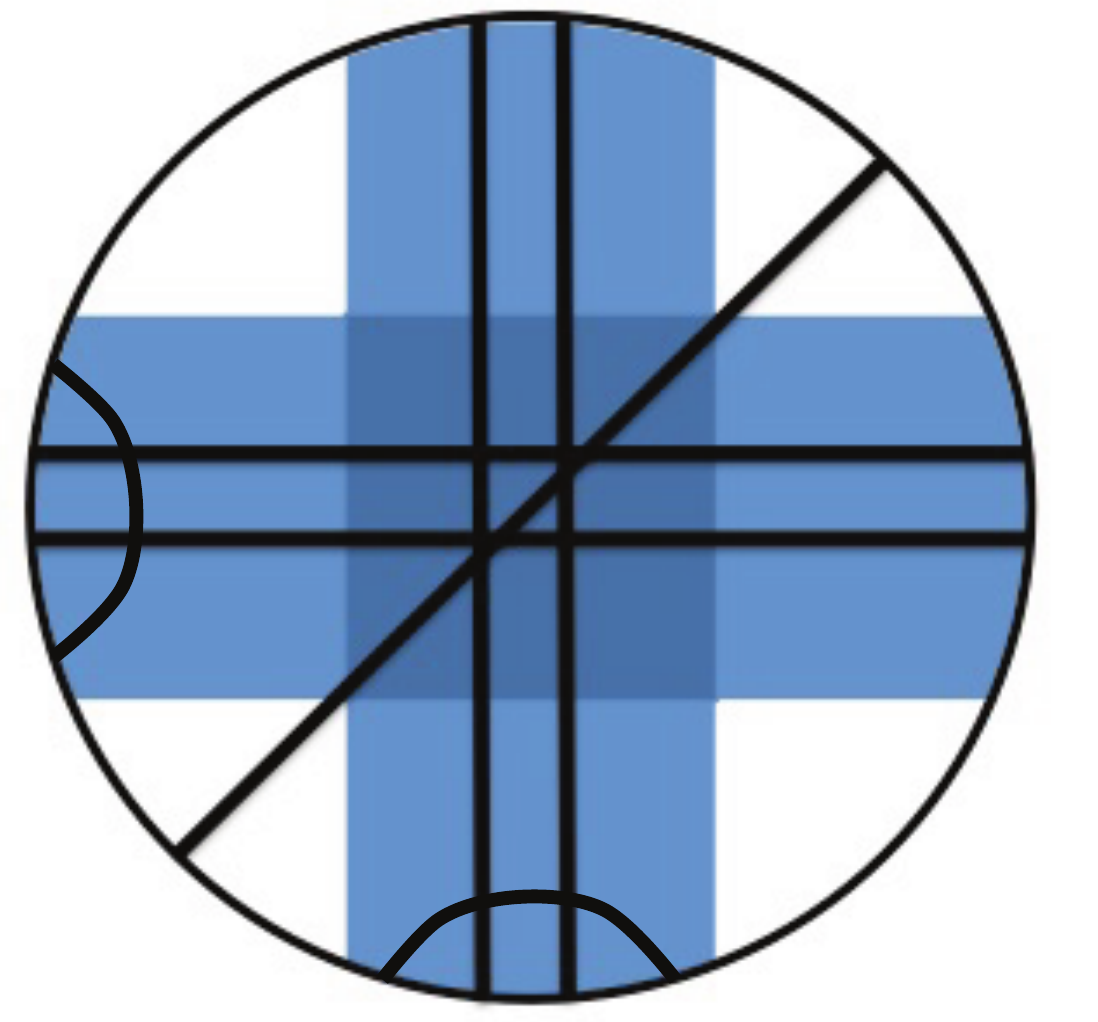}\\
\bigskip

\includegraphics[height=1.4in]{PostflypeShadow.jpg}\hspace{.25in}\includegraphics[height=1.5in]{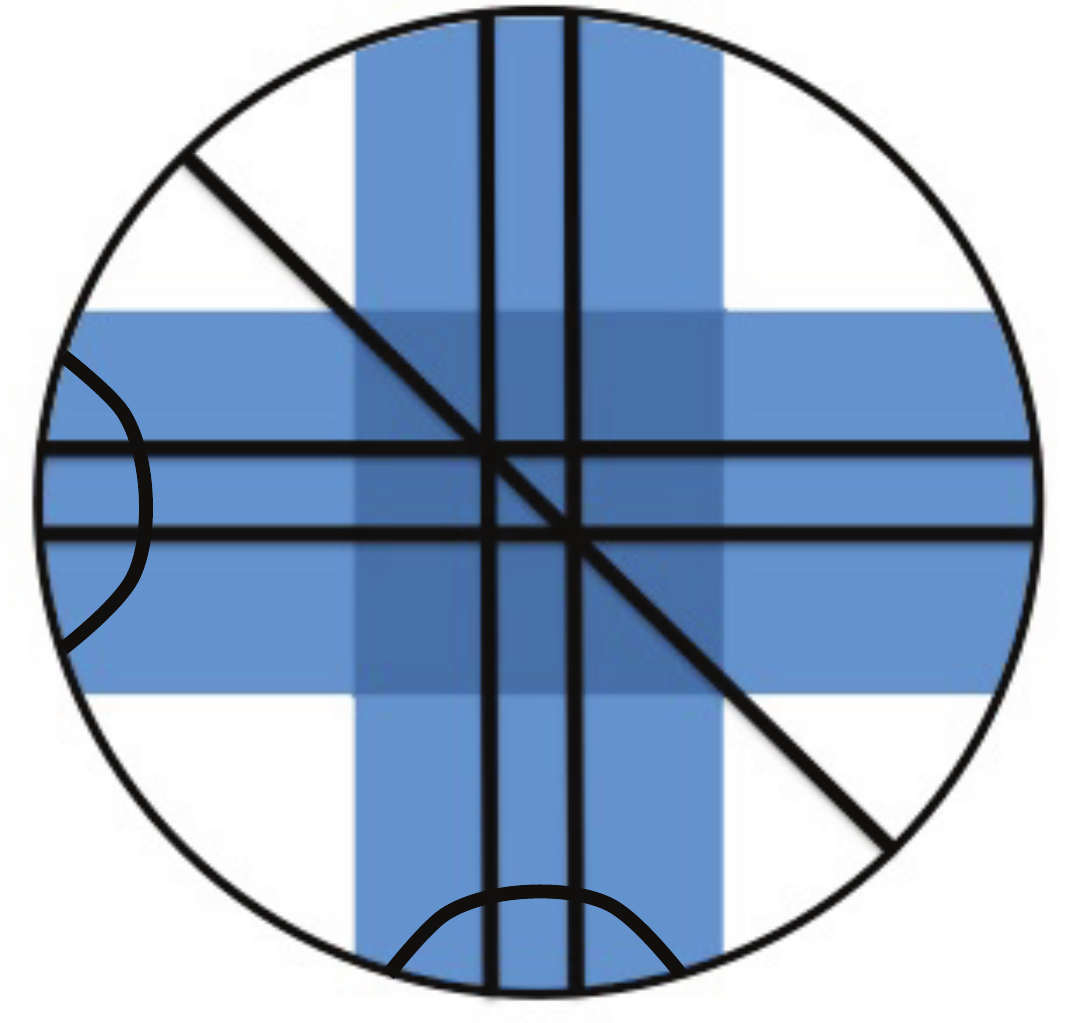}
\caption{{\bf Shadows $P_1$ and $P_2$ together with their Gauss diagrams.}}
\label{flypegauss2}
\end{center}
\end{figure}

Finally, let us observe that the $P_1$, $P_2$ counterexample can be generalized to an infinite family of counterexamples, illustrated in Figure~\ref{family}. Note that both $m$ and $n$ must be even for these pseudoknots to be non-virtual. The values of $\mathcal{I}$ corresponding to these pseudoknots can be obtained from the chord diagrams shown simply by decorating each chord with a 0.

\begin{figure}[htbp]
\begin{center}
\includegraphics[height=1.5in]{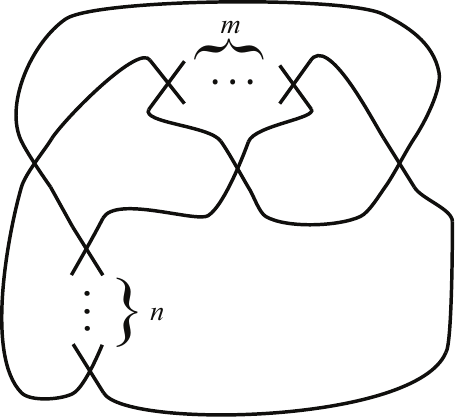}\hspace{.5in}\includegraphics[height=1.5in]{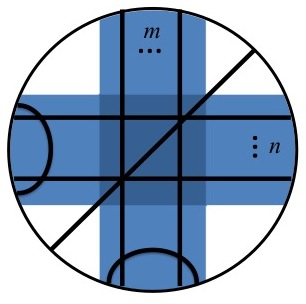}\\
\bigskip

\includegraphics[height=1.5in]{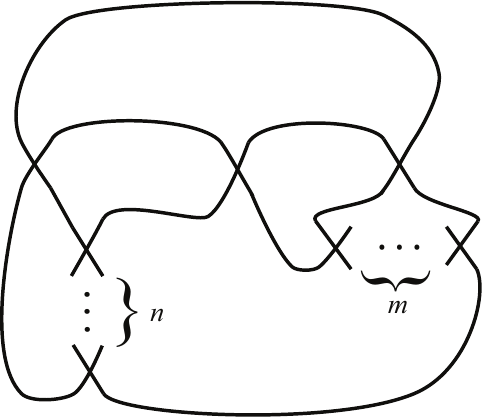}\hspace{.4in}\includegraphics[height=1.5in]{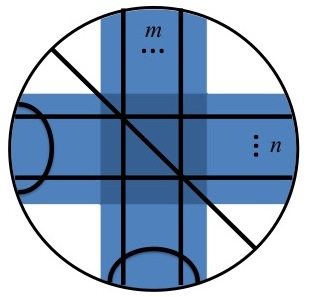}
\caption{{\bf An infinite family of pairs of counterexamples.}}
\label{family}
\end{center}
\end{figure}

\subsection{Related Combinatorial Questions}

The one open question related to this work that we are most interested in considering in the future is how many {\em non-equivalent} shadows can be obtained from a given shadow by shadow-flyping. Symmetries within chord diagrams may have the effect of producing equivalent pre- and post-flype diagrams, so not every possible flype move produces a new pseudoknot. This is an area for future exploration.



\begin{thebibliography}{00}

\bibitem{gauss}{
F. Dorais, A. Henrich, S. Jablan, and I. Johnson.
\textit{Isotopy and Homotopy Invariants of Classical and Virtual Pseudoknots} arXiv:1311.3658 (2013).}

\bibitem{hanaki}{
R. Hanaki.
\textit{Pseudo diagrams of knots, links and spatial graphs.}
Osaka J. Math.
\textbf{47} (2010), 863-883.}

\bibitem{pseudoknot}{A. Henrich, R. Hoberg, S. Jablan, L. Johnson, E. Minten, and L. Radovi\'c.
\textit{The theory of pseudoknots.} J. Knot Theory Ramifications, \textbf{22} (2013) pp. 1-21.}


\bibitem{4}{ S.~V. Jablan and R. Sazdanovi\' c, {\it LinKnot -- Knot Theory by
Computer},  World Scientific, New Jersey, London, Singapore, 2007;
http://math.ict.edu.rs/.}

\bibitem{Soulie}{C. Souli\'{e}. \textit{Complete invariants of alternating knots.} arXiv:0404490 (2004).}



\end{thebibliography}
\end{document}